\documentclass{amsart}
\usepackage{amsmath, amssymb, amsthm, color}

\newtheorem{theorem}{Theorem}[section]
\newtheorem{lemma}[theorem]{Lemma}

\newtheorem{corollary}[theorem]{Corollary}
\newtheorem{prevtheorem}{Theorem}

\theoremstyle{definition}

\theoremstyle{remark}

\begin{document}

\title[Not powers in a group.]{On the number of not powers in a finite group}
%% Group authors per affiliation:
\author{William Cocke}
\address{Department of Mathematics, University of Wisconsin-Madison, Madison WI 53706, U.S.A. 
Email: cocke@math.wisc.edu}
\date{}

\begin{abstract}
Let $G$ be a finite group and let $k$ be a positive integer. We examine the relationship between structural properties of $G$ and the number of elements of $G$ that are not $k$th powers in $G$. In particular, we examine a bound on $|G|$ given by Lucido and Pournaki and classify all cases when it is strict. We also show that when $k$ is an odd prime, then either $G$ has a normal subgroup with specific properties, or $|G|$ is bounded above by a tighter function dependent on the number of not $k$-th powers of $G$. 
\end{abstract}


\medskip


\maketitle

\section{Introduction}

Let $G$ be a finite group and $k>0$ an integer. Write

\[
G^{k}=\{x^{k}:x\in G\} \hspace{3mm} \text{ and } \hspace{3mm} \mathcal{N}_k(G)=G\setminus G^{k}.
\] We write $n_k(G)=|\mathcal{N}_k(G)|$, so $n_k(G)$ is the number of non-$k$th-powers of $G$. In recent work the author in collaboration with Isaacs and Skabelund \cite{CIS} investigated the relationship between $n_k(G)$ and $|G|$. One of their main results is:

\begin{prevtheorem}[\cite{CIS} Theorem B] \label{bound}
Let $G$ be a finite group, and write $n=n_k(G)$. If $n>0$, then $|G|\leq n(n+1)$ and in fact $|G|\leq n^2$, except in the case where $G$ is a Frobenius group with kernel of order $n+1$ and $\mathcal{N}_k(G)$ is the set of nonidentity elements of the Frobenius kernel.
\end{prevtheorem} 

Surprisingly the bound in Theorem \ref{bound} is entirely  independent of $k$. A slightly looser bound was given previously in work of Bannai et al.\ \cite{Bannai} and by L{\'e}vai and Pyber \cite{Levai}. More recently, Lucido and Pournaki \cite{LucidoI,LucidoII} investigated the proportion of elements of $G$ that are squares, and provided another proof that $n_2(G)\geq \lfloor \sqrt{G} \rfloor$ for the case $k=2$. 

It was noted in the work of Lucido and Pournaki \cite{LucidoII} that if $G$ is not one of the exceptional cases to Theorem \ref{bound}, then the bound $|G|\leq n_2(G)^2$ is strict as exhibited by the cyclic group of order 4. In this note, we prove this example is unique:

\begin{prevtheorem} \label{newbound}
If $G$ is a finite group and $|G|=n_k(G)^2$ for some $k$, then $k\equiv 2 \pmod {4}$ and $G\cong C_4$. 
\end{prevtheorem}

Restricting our attention to odd primes, we also prove the following specialized version of Theorem \ref{bound}.

\begin{prevtheorem} \label{new_jumps}
Let $G$ be a finite group, and write $n=n_p(G)$, where $p$ is an odd prime dividing $|G|$. Then $G$ satisfies one of the following statements:
\begin{itemize}
\item[(1)] $|G|=n(n+1)$ and $G$ is a Frobenius group as in Theorem \ref{bound}.
\item[(2)] $|G|=\frac{n}{2}(n+2)$ and $G$ is a central extension of a Frobenius group of order $\frac{n}{2}\left(\frac{n}{2}+1\right)$.
\item[(3)] $|G|=\frac{n}{2}(n+1)$ and $G$ is a Frobenius group with kernel of order $n+1$, and $\mathcal{N}_p(G)$ is the set of nonidentity elements of the Frobenius kernel.
\item[(4)] $|G|\leq \frac{n^2}{2}.$
\end{itemize}
\end{prevtheorem}

We note that when $|G|=\frac{n}{2}(n+2)$, then $G$ is a central extension of one of the exceptional groups in Theorem \ref{bound}. 

In Section \ref{lemmas} we will examine various inequalities regarding $n_k(G)$. Sections \ref{newbound-proof} and \ref{new_jumps-proof} contain the proofs of Theorem \ref{newbound} and Theorem \ref{new_jumps} respectively. 

\section{Inequalities involving $n_k(G)$.} \label{lemmas} 
The next lemma follows from the fact that if $\langle x \rangle = \langle y \rangle$ then $x\in \mathcal{N}_k(G)$ if and only if $y\in \mathcal{N}_k(G)$. 

\begin{lemma}\label{divisible}
Let $p$ be a prime. If $n=n_p(G)$ for a finite group $G$, then $p-1$ divides $n$. 
\end{lemma}

\begin{proof}
Let $\sim$ be the equivalence relation $x \sim y$ iff $\langle x \rangle = \langle y \rangle$. We can partition $\mathcal{N}_p(G)$ into equivalence clases under $\sim$. Each such equivalence class has size divisible by $p-1$ and thus $n=|\mathcal{N}_p(G)|$ is divisible by $p-1$.   
\end{proof}

The following lemma is somewhat suprising. Let $p$ be a prime. For $H\leq G$ the set $\mathcal{N}_p(H)$ is not always a subset of $\mathcal{N}_p(G)$. However $n_p(H)\leq n_p(H)$. 

\begin{lemma} \label{in and out} Let $G$ be finite a group and $H$ be a subgroup of $G$. Let $p$ be a prime. Then $n_p(H)\leq n_p(G)$. Moreover $n_p(H)=n_p(G)$ if and only if $\mathcal{N}_p(H)=\mathcal{N}_p(G)$.  
\end{lemma}

\begin{proof}
For $x\in G^{p}$, write $\theta(x)=\{y\in G: y^{p}=x\}$, and note that by assumption the sets $\theta(x)$ are nonempty and disjoint, and their union is the whole group $G$. It follows that
\[
n_p(G)=|G|- |G^{p}|=\left (\sum_{x\in G^{p}} |\theta(x)|\right )- |G^{p}|=\sum_{x\in G^{p}}(|\theta(x)|-1).
\]
Similarly, if $x\in H^{p},$ we write $\varphi(x)=\{y\in H:y^{p}=x\}$. Then 
\[
n_p(H)=\sum_{x\in H^{p}}(|\varphi(x)|-1).
\]

Now $H^{p}\subseteq G^{p}$ and for $x\in H^{p}$ we have $\varphi(x)=H\cap\theta(x)$, so $|\varphi(x)|\leq |\theta(x)|.$ Noting that the terms $|\theta(x)|-1$ are nonnegative for $x\in G^{p}\setminus H^p$ we have:
\begin{equation}\label{ineqeq}
n_p(G)-n_p(H)\geq \sum_{x\in H^{p}}(|\theta(x)|-|\varphi(x)|)\geq 0.
\end{equation}

Hence $n_p(H)\leq n_p(G)$. If $n_p(H)=n_p(G)$, then both \[\sum_{\substack{x \in G^{p}\\ x \not \in H^{p}}} |\theta(x)| -1 =0 \text{\;  and } \sum_{x\in H^{p}}(|\theta(x)|-|\varphi(x)|) =0.\] Thus every element of $H^{p}$ has the same number of $p$th roots in $H$ as it does in $G$ and all elements of $G^{p}$ that are not in $H^{p}$ have order not divisible by $p$. 

If $n_p(G)=n_p(H)$ then $x\in \mathcal{N}_p(H)$ implies that $x\in \mathcal{N}_p(G)$ and thus $\mathcal{N}_p(G)=\mathcal{N}_p(H)$. 
\end{proof}

%Recall that $\textbf{O}^{\pi}(G),$ for a set of primes $\pi$, is the characteristic subgroup of $G$ satisfying $G/\textbf{O}^{\pi}(G)$ is a $\pi$-group and $M\supseteq \textbf{O}^{\pi}(G)$ whenever $M\lhd G$ and $G/M$ is a $\pi$-group. 
The following corollaries demonstrate some implications of $n_p(H)=n_p(G)$ for $H\subseteq G$ and $G$ finite:

\begin{corollary}\label{Sylow_restricts}
Let $p$ be a prime and let $G$ be a finite group. Suppose $H<G$ with $n_p(H)= n_p(G)$. Then $\textbf{O}^{p'}(G)\subseteq H$; in particular every Sylow $p$-subgroup of $G$ is contained in $H$.
\end{corollary}

\begin{proof}
The set $X=\{x\in G: o(x)=p^k, k\in \mathbb{N} \}$ generates $\textbf{O}^{p'}(G)$. Since every element of order $p^k$ is contained in $\langle y \rangle$ for some $y\in \mathcal{N}_p(G)$, we conclude that \[\textbf{O}^{p'}(G)=\langle X \rangle \subseteq \langle \mathcal{N}_p(G) \rangle =\langle \mathcal{N}_p(H)\rangle \subseteq H.\] If $S\in \text{Syl}_p(G)$ then $S\subseteq \textbf{O}^{p'}(G)$.
\end{proof}

%\begin{corollary}
%Let $p$ be a prime and let $H\subseteq G$, $G$ a finite group. If $n_p(H)=n_p(G)$ then every Sylow $p$-subgroup of $G$ is contained in $H$. \end{corollary}

%\begin{proof}
%Let $S$ be a Sylow $p$-subgroup of $G$. Then $S\subseteq \textbf{O}^{p'}(G)\subseteq \langle \mathcal{N}_p(G)\rangle$.
%\end{proof}

\begin{lemma} \label{center} Let $G$ be a finite group, and suppose that $p$ divides $|\textbf{Z}(G)|$, where $p$ is a prime. Then \[|G|\leq \frac{p\, n_p(G)}{p-1},\] and if equality holds then $G$ has a normal cyclic Sylow $p$-subgroup. 
\end{lemma}

\begin{proof}
Let $Z\subset \textbf{Z}(G)$ have order $p$. Since all elements in each coset of $Z$ in $G$ have the same $p$th power, it follows that $|G^{p}|$ is at most the number of cosets of $Z$ in $G$, i.e., $|G:Z|=|G|/p$. Then
\[
n_p(G)=|G|- |G^{p}|\geq |G|-\frac{|G|}{p}=\frac{p-1}{p}|G|.
\]

If $n_p(G)=\frac{p-1}{p}|G|,$ then every coset of $Z$ has a unique $p$th power. As in the proof of Lemma \ref{in and out}, for $x\in G^{p}$ write $\theta(x)=\{y\in G: y^{p}=x\}$. If $n_p(G)=\frac{p-1}{p}|G|$, then $\theta(x)$ is a single coset of $Z$, and thus $|\theta(x)|=p$ for all $x\in G^{p}$. 

Consider the set 
\[
S=\{x\in G: o(x)=p^{k}, k\in \mathbb{Z}\}. 
 \] We claim $S$ is a normal cyclic Sylow $p$-subgroup of $G$. Let $s\in S$ have maximum order. We claim that $S=\langle s \rangle.$ Suppose that $x\in S$ has minimal order such that $x\notin \langle s \rangle.$ Then $x^{p}\in \langle s \rangle$ and $|\langle s \rangle\cap \theta(x^{p})|=p$. But, $|\theta(x^{p})|=p$. Hence $x\in \theta(x^{p})\subseteq \langle s \rangle$. %Let $e$ be the smallest integer such that $s^{e}=1$ for all $s\in S$, i.e., $e$ is the largest order of an element of $S$. Since $Z\subseteq S$ we have that $e\geq p$. Let $s\in S$ have order $e$. Then $\langle s  \rangle$ is a cyclic subgroup of $G$ of order $e$. In a cyclic group $H$, every element of $H^{p}$ has exactly $p$, $p$th roots. Since $Z\subseteq S$ and all elements of order $p$ in $G$ are contained in $Z$, we conclude that $Z\subset \langle s \rangle$. Hence any element of the cyclic group $\langle s \rangle$ of order less than $e$ has exactly $p$, $p$th roots and the $p$th roots form a coset of $Z$. 

Therefore $S = \langle s \rangle$. 
\end{proof}

As part of our proof of Theorem \ref{newbound} we will see that the proportion of non-$k$th-powers under the action of taking quotients behave nicely:

\begin{lemma}\label{quotient}
Let $G$ be a finite group. If $k>0$ and $N$ is a normal subgroup of $G$ then
\[
\frac{n_k(G/N)}{|G/N|} \leq \frac{n_k(G)}{|G|},
\]  with equality if and only if for all $x\in G$ every coset representative of $x^{k}N$ is in $G\setminus \mathcal{N}_k(G)$. 
\end{lemma}

%\begin{proof}
%Let $\pi:G\rightarrow G/N$ be the natural map taking $g\in G$ to $gN$. We note that $\pi^{-1}(\mathcal{N}_k(G/N))$ is a subset of $\mathcal{N}_k(G)$ of order $|N|$. Moreover, if $\pi(g)\in \mathcal{N}_k(G/N)$, then $\pi^{-1}(\pi(g))\subseteq \mathcal{N}_k(G)$. Hence \begin{equation} \label{quotienteqn}|\mathcal{N}_k(G/N)|\cdot |N| \leq |\mathcal{N}_k(G)|\end{equation} with equality if and only if $\pi$ maps every non-$k$th-power of $G$ to a non-$k$th-power of $G/N$.

%Dividing both sides of (\ref{quotienteqn}) by $|G|$ gives us the desired result. 
%\end{proof}

%We will write $s_k(G)=|G^{k}|$; that is $s_k(G)$ is the number of $k$th powers of $G$. We can rewrite Lemma \ref{quotient} as  

%\begin{lemma} \label{quotient2}
%Let $G$ be a finite group. If $k>0$ and $N$ is a normal subgroup of $G$ then 
%\[ \frac{s_k(G)}{|G|} \leq \frac{s_k(G/N)}{|G/N|}. \] \end{lemma}

We now return our attention, for the moment, to the case $k=p$. 

\begin{theorem} \label{p-groups-n}
Let $G$ be a finite $p$-group of order $p^m$, and write $n=n_p(G)$. If $G$ is cyclic then $n=p^m-p^{m-1}$. Otherwise $n\geq p^m-p^{m-2}$. 
\end{theorem}

\begin{proof}
If $G$ is cyclic, then the only elements of $G$ in $\mathcal{N}_p(G)$ are the elements with order equal to $|G|$. Hence $n=\varphi(p^m)=p^{m}-p^{m-1}$. 

If $G$ is not cyclic then $G/\Phi(G)$ is not cyclic and $G$ has a normal subgroup $F$ such that $G/F$ is elementary abelian of rank 2. By Lemma \ref{quotient} we see that  
\[
\frac{n}{|G|}\geq \frac{n_p(G/F)}{|G:F|} = \frac{n_p(C_p\times C_p)}{p^2} = \frac{p^2-1}{p^2}.
\]

Hence $n \geq p^m-p^{m-2}.$
\end{proof}

We will now introduce some notation. For a prime $p$, the set $\mathcal{N}_p(G)$ of non-$p$-th powers of $G$ is a union of conjugacy classes of $G$. Write \[\mathcal{N}_p(G)=x_1^G\cup \dots\cup x_m^{G}.\] Without loss of generality, we will assume that the listing of conjugady classes is ordered so that $o(x_i)\leq o(x_j)$ whenever $i\leq j$. The \textbf{type} of $\mathcal{N}_p(G)$ is the $m$-tuple $(o(x_1),\dots,o(x_m))$. We will refer to $m$ as the \textbf{length} of $\mathcal{N}_p(G)$. 

Recall that an element $y$ of a group $G$ is said to be $p$-singular if $p$ divides the order of $y$. 

\begin{lemma}\label{propagation}
Let $G$ be a finite group. Let $m$ be the length of $\mathcal{N}_p(G)$. Let $Y$ denote the set of orders of p-singular elements of $G$. Let $X$ be the set of integers $j$ such that $p^{k} j\in Y$ and $\gcd(j,p)=1$. Then $|X|\leq m$.
\end{lemma}

\begin{proof}
We know that for each $a\in X$ there is an element $y\in G$ such that $o(y)=j$.  Let $z\in \mathcal{N}_p(G)$ such that $z^{p^k}=y$, with $k$ minimal. Then $o(z)=p^k\cdot j$. Since $z$ depends on $j$, we conclude that distinct $i,j\in X$ will yield distinct elements $z_j,z_i$. Since $o(z_j)\neq o(z_i)$ we conclude that $z_j^{G}\neq z_i^{G}$.
\end{proof}

We will use Lemma \ref{propagation} to analyze groups for which the length of $\mathcal{N}_p(G)$ is small. 

%%%%%
%%%%%
%%%%%
%%%%%
%%%%%

\section{Proof of Theorem \ref{newbound}} \label{newbound-proof}

In this section, we will prove that the only group $G$ for which there is an integer $k$ such that $|G|=n_k(G)^2$ is $C_4$. Recall the following lemma:

\begin{lemma}[Lemma 2.5 \cite{CIS}] Let $G$ be a group with $0<n_k(G)<\infty$. Then there exists a prime $p$ dividing $k$ such that $0<n_p(G)\leq n_k(G)$. 
\end{lemma}

We immediately have: 
\begin{corollary}\label{k-to-p} If $|G|=n_k(G)^2$ and $G$ is finite, then $|G|=n_p(G)^2$ for some prime $p$ dividing $k$. 
\end{corollary}

In the rest of the section we will prove that $|G|=n_p(G)^2$ for a prime $p$, if and only if $p=2$ and $G\cong C_4$. 

\begin{lemma}\label{cases}
Let $G$ be a finite group and write $n=n_p(G)$ for a prime $p$. If $|G|=n^2$ and $m$ is the length of $\mathcal{N}_p(G)$, then one of the following holds:
\begin{itemize}
\item[(1):] $m=1$.
\item[(2):]$p=2$ and $m=2$.
\end{itemize}
\end{lemma}

\begin{proof}
There is some $x\in \mathcal{N}_p(G)$ with $|x^{G}|\leq n_p(G)/m$. 
Moreover, $x\in \textbf{Z}(\textbf{C}_G(x))$ and $p$ divides $o(x)$. We conclude that 
\[
n_p(G)^2=|G|=|x^{G}||\textbf{C}_G(x)|\leq \frac{n_p(G)}{m} \frac{p}{p-1} n_p(G).
\]

Hence $(p-1) m \leq p$ and we conclude that either $m=1$; or $p=2$ and $m=2$. 
\end{proof}

\begin{theorem}\label{only_one_class_theorem}
Let $G$ be a finite group with $n=n_p(G)>0$. If $\mathcal{N}_p(G)$ has length 1, then $|G|\neq n^2$. 
\end{theorem}

\begin{proof}
We first note that $n$ must be greater than 1, since $n=1$ implies that $G=C_2$ by Theorem \ref{bound}. 

By way of contradiction assume that $|G|=n^2$. There is an $x\in \mathcal{N}_p(G)$ of order $p^k$, for some positive integer $k$. Since the length of $\mathcal{N}_p(G)$ is $1$ we conclude that all non-$p$th powers in $G$ have order $p^k$. Consider $C=\textbf{C}_G(x)$. Lemma \ref{propagation} shows that $C$ is a $p$-group, or else the length of $\mathcal{N}_p(G)$ would be greater than 1. Let $|C|=p^{j}$ for some $j\geq k \geq 1$.

We now have that $|G|=n^2=|x^{G}|\cdot|C|=n\cdot p^{j}$. Hence $n=p^j$ and $G$ is a $p$-group of order $p^{2j}$. Theorem \ref{p-groups-n},  gives us 
\[
p^j=n\geq (p^2-1)p^{2j-2}.\]
Dividing both sides by $p^j$ gives us \[1 \geq (p^2-1)p^{j-2},\] and thus $j=1$ and we conclude that $|G|=p^2$. A contradiction to $|x^{G}|=n>1.$
\end{proof}

\begin{theorem}
If $G$ is a finite group and $|G|=n_k(G)^2$ for some $k$, then $G\cong C_4$. 
\end{theorem}
\begin{proof}
Let $G$ be a finite group satisfying $|G|=n_k(G)^2$ for some $k$. Then $|G|=n_p(G)^2$ for some prime $p$ dividing $k$ by Corollary \ref{k-to-p}. Furthermore by Lemma \ref{cases} and Theorem \ref{only_one_class_theorem}, we may assume that the length of $\mathcal{N}_k(G)$ is exactly 2, and that $p=2$. Let $n=n_2(G)$.

Let $x\in \mathcal{N}_2(G)$ such that $|x^{G}|$ is minimal, and write $C=\textbf{C}_G(x)$. We have

\[n^2=|G|=|x^{G}||C|\leq \frac{n}{2} 2 n_2(C))\leq n^2.\] Therefore, we must have the following equalities: $|x^{G}|=\frac{n}{2}$, $n_2(C)=n$, and $|C|=2 n$ for any $x\in \mathcal{N}_2(G)$.  By Lemma \ref{in and out}, we are guaranteed that $\mathcal{N}_2(C)=\mathcal{N}_2(G)$. Moreover, by Lemma \ref{center}, we know that the Sylow $2$-subgroup of $C$ is cyclic. Now, fix an $x\in \mathcal{N}_2(G)$ such that $o(x)=2^{j}$. Because $o(x)$ is $2^j$ and $x$ is not a square in $C$, we see that the Sylow $2$-subgroup of $C$ is generated by $x$ and has order $2^j$. By Corollary \ref{Sylow_restricts}, $\langle x \rangle$ is a Sylow subgroup of $G$.  Moreover, $\langle x \rangle$ is normal in $G$. 

Lemma \ref{propagation} further tells us that $C$ can be divisible by at most one odd prime. Let $|C|=2^{j} q^{\ell}$. Then $n=2^{j-1}q^{\ell}$ and $|G|=2^{2j-2} q^{2\ell}$. Since $\langle x \rangle$ is a Sylow $2$-subgroup of $G$, we see that $2^{2j-2}=2^{j}$ and thus $j=2$; moreover, $G$ has a cyclic Sylow $2$-subgroup and thus has a normal $2$-complement $H$. Since normal subgroups commute, $H\subset C=\textbf{C}_G(x)$. We conclude that $\ell=2\ell$. Hence $\ell=0$ and we conclude that $|G|=4$ and $G$ is cyclic. 
\end{proof}

%%%%%
%%%%%
%%%%%
%%%%%
%%%%%

\section{Proof of Theorem \ref{new_jumps}}\label{new_jumps-proof}

To prove Theorem \ref{new_jumps} we will first examine how the type of $\mathcal{N}_p(G)$ gives a bound on the order of $G$.  

\begin{lemma}\label{number_of_classes}
Let $G$ be a finite group and write $n=n_p(G)$ for a prime $p$. If $|G|>\frac{n^2}{2}$ and $m$ is the length of $\mathcal{N}_p(G)$, then either $m\leq 2$ or $p=2$ and $m=3.$
\end{lemma}

\begin{proof}
There is some $x\in \mathcal{N}_p(G)$ with $|x^{G}|\leq n_p(G)/m$. 
Moreover, $x\in \textbf{Z}(\textbf{C}_G(x))$ and $p$ divides $o(x)$. We conclude that 
\[
\frac{n^2}{2}<|G|=|x^{G}||\textbf{C}_G(x)|\leq \frac{n}{m} \frac{p}{p-1} n.
\]

Hence $(p-1) m < 2 p$ and we conclude that either $m\leq2$; or $p=2$ and $m=3$. 
\end{proof}

\begin{lemma}\label{exponent}
Let $G$ be a finite group and $p$ a prime. If $G$ contains an element of order $p^k$ for $k>1$, then $|G|\leq \frac{n_p(G)^2}{p^{k-2}(p-1)}.$
\end{lemma}

\begin{proof}
Let $S$ be a Sylow $p$-subgroup of $G$ and let $p^{k}$ be the exponent of $S$. Suppose that $k>1$. We will show that $|G|\leq \frac{n_p(G)^2}{p^{k-2}(p-1)}$.  

Let $K$ be the set of all elements of $G$ of order $p^{k}$. Then $K\subseteq \mathcal{N}_p(G)$ and is a normal subset of $G$. Consider the set $K^{p^{k-1}}$ of $p^{k-1}$ powers of elements of $K$. Let $\mu:G\rightarrow G$ take $x\rightarrow x^{p^{k-1}}$. For an element $y\in K$, we see that $\mu(y)\in K^{p^{k-1}}$; moreover, $\mu$ is at least $p^{k-1}:1$ from $\langle y \rangle$ to $y^{p^{k-1}}$. Hence $|K^{p^{k-1}}|\leq \frac{|K|}{p^{k-1}}$. Therefore 
\[
|G|=\left|(y^{p^{k-1}})^G\right|\cdot\left|\textbf{C}_G(y^{p^{k-1}})\right| \leq \frac{|K|}{p^{k-1}} \frac{p}{p-1}n_p(G)\leq \frac{n_p(G)^2}{p^{k-2}(p-1)}.
\]
\end{proof}

As seen in both Lemma \ref{number_of_classes} and \ref{exponent} the prime 2 is special and will often require a separate argument. Recall that the type of $\mathcal{N}_p(G)$ is a list of the orders of conjugacy classes in $\mathcal{N}_p(G)$. By combining Lemmas \ref{number_of_classes}, \ref{exponent}, and \ref{propagation} we can greatly restrict the type of $\mathcal{N}_p(G)$ in the case that $p$ is an odd prime and $|G|>\frac{n_p(G)^2}{2}$.

\begin{corollary}\label{odd-cases}
Let $G$ be a finite group and $p$ an odd prime dividing $|G|$. Write $n=n_p(G)$ and let $m$ be the number of conjugacy classes of $G$ contained in $\mathcal{N}_p(G)$. If $|G|>\frac{n^2}{2}$, then the type of $\mathcal{N}_p(G)$ is either $(p)$,$(p,p)$ or $(p,qp)$. 
\end{corollary}

%\begin{proof}
%As noted in Lemma \ref{number_of_classes} we can assume that the length of $\mathcal{N}_p(G)$ is less than or equal to $2$. If $|G|>n^2/2$, then since $p$ is odd, Lemma \ref{exponent} implies that $G$ contains no elements of order $p^2$. Since an element of order a power of $p$ is always contained in $\mathcal{N}_p(G)$, we conclude that if the length of $\mathcal{N}_p(G)$ is 1, then the type of $\mathcal{N}_p(G)=(p)$. If the length of $\mathcal{N}_p(G)$ is 2, we know there is an $x\in \mathcal{N}_p(G)$ of order $p$ and a $y\in \mathcal{N}_p(G)\setminus x$. Since $G$ contains no elements of order $p^2$, we know that $p^2$ does not divide $o(y)$. If two primes $q,r$ distinct from $p$ divide the order of $(y)$, then $G$ contains non-$p$th powers of order $qp$ and $qrp$, as well as non-$p$th powers of order $p$, contradicting the length of $\mathcal{N}_p(G)$. 
%\end{proof}

Of course there is a corresponding classification for the case $p=2$, but the parametrization of possible types is not as succinct. 

The following theorem of Frobenius will be used later to obtain appropriate bounds. A nice, self-contained proof can be found in a note by Isaacs and Robinson \cite{IR}.

\begin{lemma}[Frobenius's Solution Theorem]
If $m$ divides $|G|$, then $m$ divides \[|\{x\in G: x^m = 1\}|. \] 
\end{lemma}

If $p$ is an odd prime then we have the following theorem classifying when the type of $\mathcal{N}_p(G)$ is $(p)$ in Corollary \ref{odd-cases}:

\begin{theorem} \label{odd-type-1}
Let $G$ be a finite group and $p$ an odd prime dividing the order of $G$. Write $n=n_p(G)$. If the type of $\mathcal{N}_p(G)$ is $(p)$ and $|G|\neq n(n+1)$, then $|G|\leq \frac{n(n+1)}{3}.$ 
\end{theorem}

\begin{proof}
Let $x\in \mathcal{N}_p(G)$. By Lemma \ref{propagation}, $\textbf{C}_G(x)$ is a $p$-group and is contained in a Sylow $p$-subgroup of $S$. Since $\mathcal{N}_p(G)$ has type $(p)$ we know that all nontrivial elements of $S$ are in $\mathcal{N}_p(G)$ and hence conjugate to $x$. Let $y\in \textbf{Z}(S)\setminus 1$. Let $C=\textbf{C}_G(y)$. Since $y\in \mathcal{N}_p(G)$, we have that $C=S$. Let $|C|=|S|=p^k$. We know that $p^k$ divides $|G|$. By the theorem of Frobenius, $p^k|(n+1)$. We have 
\[
|G|=|x^{G}||C|=n\cdot p^k.
\]

If $p^k=n+1,$ then $|G|=n(n+1)$. Otherwise suppose $p^k=\frac{n+1}{2}$ and $n=2p^k-1$. By Lemma \ref{divisible}, we know that $n$ is divisible by $p-1$ which is even since $p$ is an odd prime; This contradicts $n=2p^{k}-1$. Therefore if $p^k\neq n+1$, then $p^k \leq \frac{n+1}{3}.$ 
\end{proof}

We note that 
\[
\frac{n(n+1)}{3} \leq \frac{n^2}{2},
\] when $n\geq 2$. When $n=1$, $|G|\leq 2$ by Theorem \ref{bound} and hence no odd primes divide $|G|$.

We now handle the two remaining cases in Corollary \ref{odd-cases}.

\begin{theorem} \label{odd-type-2}
Let $G$ be a finite group and $p$ an odd prime dividing the order of $G$. Write $n=n_p(G)$. Assume  $\mathcal{N}_p(G)$ has length 2. If $|G| > \frac{n^2}{2}$ then one of the following happens:
\begin{itemize}
\item The type of $\mathcal{N}_p(G)$ is $(p,p)$ and $|G|=\frac{n}{2}(n+1)$ and $G$ is a Frobenius group.
\item The type of $\mathcal{N}_p(G)$ is $(p,2p)$ and $|G|=\frac{n}{2}(n+2)$ and $G$ is a central extension of a Frobenius group of order $\frac{n}{2}(\frac{n}{2}+1)$. 
\end{itemize}
\end{theorem}

\begin{proof}
By Corollary \ref{odd-cases}, we know that the type of $\mathcal{N}_p(G)$ is either $(p,p)$ or $(p,pq)$ for $q$ a prime. We proceed by cases. 

Suppose that the type of $\mathcal{N}_p(G)$ is $(p,p)$. Let $x,y$ be elements of $\mathcal{N}_p(G)$ in different conjugacy classes. Without loss of generality assume that $|x^{G}|\leq |y^{G}|$, so $|x^{G}|\leq \frac{n}{2}$. By Lemma \ref{propagation}, we know that $C=\textbf{C}_G(x)$ is a $p$-group and thus $|C|\leq (n+1)$. Therefore:
\[|G|=|x^{G}||C|\leq |x^{G}| (1+n)\leq \frac{n}{2}(n+1).\]
If $|x^{G}|<\frac{n}{2}$, then since $n$ is even by Lemma \ref{divisible}, we know that $|x^{G}|\leq \frac{n}{2}-1$ and thus \[|G|\leq (\frac{n}{2}-1)(n+1)\leq \frac{n^2}{2}.\] Suppose $|x|=\frac{n}{2}$ and that $|C|<(n+1)$. Then $|C|\leq n$ and \[|G|\leq \frac{n^2}{2}.\] 

Hence if the type of  $\mathcal{N}_p(G)$ is $(p,p)$ and $|G|>\frac{n^2}{2}$ then $|G|=\frac{n}{2}(n+1)$. If we have $|G|=\frac{n}{2}(n+1)$, then for all $x\in \mathcal{N}_p(G)$ we have $|x^G|=\frac{n}{2}$ and $C=\textbf{C}_G(x)$ has order $n+1$. Moreover $C=\mathcal{N}_p(G)\cup 1$ is a normal subgroup of $G$. We further note that $n+1$ and $\frac{n}{2}$ are coprime, so by the Schur--Zassenhaus theorem $C$ has a complement in $G$. Since the centralizer of any nontrivial element of $C$ is contained in $C$, we see that $G$ is a Frobenius group with Frobenius kernel $C$ consisting of $\mathcal{N}_p(G)$ together with the identity. 

Suppose that the type of $\mathcal{N}_p(G)$ is $(p,pq)$ for some prime $q$. Let $x\in \mathcal{N}_p(G)$ have order $p$ and $y\in \mathcal{N}_p(G)$ have order $pq$. We note that $x^G$ contains all elements of $G$ of order $p$. Hence $y^q\in x^G$. Moreover, every $q$th power of an element of $y^G$ is in $x^G$. The $q$th power map from $y^G$ to $x^G$ is $j$ to $1$ for some positive integer $j$. Since $|y^G|+|x^{G}|=n$ we have 
\[
n=(j+1)|x^{G}| \hspace{2mm} \text{ and } \hspace{2mm} |x^{G}| = \frac{n}{j+1}. 
\] Now consider $C=\textbf{C}_G(x)$. Every element of $C$ has order $1,p,q$ or $pq$. We wish to bound the number of elements in $C$ of each of order. We note that there are exactly $j$ elements in $C$ of order $pq$ whose $q$th power is $x^q$. Moreover, any element $s\in C$ of order $q$ will satisfy $(xs)^q=x^q$. Hence there are at most $j$ elements $s\in C$ of order $q$. Since all elements of $y^{G}$ that have $x^q$ as their $q$-power commute with $x$, we conclude that there are exactly $j$ elements of order $q$ in $C$. We also know that there are at most $n$ elements total of orders $p$ and $pq$ in $C$. Hence $|C|\leq n+j+1$. We thus have
\[
|G|=|x^{G}||C|= \frac{n}{j+1}|C| \leq \frac{n}{j+1}(n+j+1).
\]

For $j>1$ and $n>8$, we have $\frac{n}{j+1}(n+j+1)\leq \frac{n^2}{2}$. Therefore if $|G|>\frac{n^2}{2}$ and $|G|>56$, we can assume that $j=1$. (For groups with order less than or equal to 56, we verified the theorem directly in Magma \cite{Magma}.) Since the map $q$th power map from $y^G$ to $x^G$ is 1:1, we can assume that $q=2$, otherwise $C$ would contain more than $j$ elements of order $q$. Hence the element of order $2$ in $C$ is central in $C$ (since there is only one such element). Therefore the number of elements of $G$ of order $p$ and $2p$ are equal and thus $n$ is even. Hence if $|C|<n+2$, then $|C|\leq n$ (since $|C|$ is even) and we have that 
\[
|G|=|x^{G}|C|\leq \frac{n^2}{2}.
\]

Therefore if $|G|>\frac{n^2}{2}$ and the type of $\mathcal{N}_p(G)$ is not $(p,p)$ then the type of $\mathcal{N}_p(G)$ is $(p,2p)$ and $|G|=\frac{n}{2}(n+2)$. Suppose $|G|=\frac{n}{2}(n+2)$ and let $x\in \mathcal{N}_p(G)$ satisfy $o(x)=p$. Let $C=\textbf{C}_G(x)$. Then $|C|=n+2$ and $C$ contains a unique involution $z$. Moreover, $C$ is normal, since it is generated by the normal set $\mathcal{N}_p(G)$. Hence $z$ is central in $G$, being the unique element of order $2$ in a normal subgroup of $G$. 

We now ask about the group $\overline{G}=G/\langle z \rangle.$ What is $n_p(\overline{G})$? It must be the case that $n_p(\overline{G})\leq n/2$. By Theorem \ref{bound}, we are guaranteed that $n_p(\overline{G})=\frac{n}{2}$ and $\overline{G}$ is a Frobenius group with kernel of order $\frac{n}{2}+1$ and complement of order $\frac{n}{2}$. Hence $G$ is a central extension of such a Frobenius group. 
\end{proof}

We can now prove Theorem \ref{new_jumps}:

\begin{proof}
By Corollary \ref{odd-cases} we can reduce to either the length of $\mathcal{N}_p(G)$ is 1 or 2. If the length of $\mathcal{N}_p(G)$ is 1, then Theorem \ref{odd-type-1} demonstrates that $|G|=n(n+1)$ or $|G|\leq \frac{n^2}{2}$. If the length of $\mathcal{N}_p(G)$ is 2, then by Theorem \ref{odd-type-2}, either $|G|\leq \frac{n^2}{2}$ or $G$ satisfies hypotheses (2) or (3) of the theorem.
\end{proof}

\section{Acknowledgements}
This material is based upon work supported by the National Science Foundation under Grant No. DMS-1502553. Some of the work was done while the author was visiting the Army Cyber Institute. The views expressed are those of the author and do not reflect the official policy or position of the Army Cyber Institute, West Point, the Department of the Army, the Department of Defense, or the US Government.

\bibliographystyle{alpha}
\newcommand{\etalchar}[1]{$^{#1}$}

\end{document}